\documentclass{amsart}
\usepackage{color}
\usepackage{amssymb}
\usepackage{amsmath}
\usepackage{pdfsync}
\usepackage{ulem}

\def\CX{{\mathbb C}}

\def\PX{{\mathbb P}}

\def\GL{{\rm GL}}

\def\gl{{\rm gl}}
\def\Scal{{\rm Scal}}
\def\diag{{\rm Diag}}
\def\d{{
\partial}}

\def\calD{{\mathcal D}}

\def\calO{{\mathcal O}}
\def\calP{{\mathcal P}}
\def\d{{
\partial}}

\def\tbar{{t}}
\def\Hbar{{\overline{H}}}
\def\Gbar{{\overline{G}}}

\def\pp{{\mathbb{P}}^1({\mathbb{C}})}

\newtheorem{theorem}{Theorem}[section]

\theoremstyle{definition}

\newtheorem{thm}{Theorem}
\newtheorem{lem}[thm]{Lemma}
\newtheorem{cor}[thm]{Corollary}
\newtheorem{prop}[thm]{Proposition}
\newtheorem{defin}[thm]{Definition}

\newtheorem{remark}[thm]{Remark}

\numberwithin{equation}{section}

\begin{document}

\title{Projective Isomonodromy and Galois Groups}


\author{Claude Mitschi}
\address{Institut de Recherche Math\'ematique Avanc\'ee, Universit\'e de Strasbourg et CNRS, 7 rue Ren\'e Descartes, 67084 Strasbourg Cedex, France}
\email{mitschi@math.unistra.fr}
\thanks{}

\author{Michael F. Singer}
\address{Department of Mathematics, North Carolina
State University, Box 8205, Raleigh, North Carolina 27695-8205}
\curraddr{}
\email{singer@math.ncsu.edu}
\thanks{The second author was partially supported by NSF Grants CCF-0634123 and CCF-1017217. He would also like to thank the   Institut de Recherche Math\'ematique Avanc\'ee, Universit\'e de Strasbourg et C.N.R.S., for its hospitality and support during the preparation of this paper.}


\date{}

\dedicatory{}

\commby{}

\begin{abstract}
In this article we introduce the notion of projective isomonodromy, which is a special type of monodromy evolving deformation of  linear differential equations, based on the example  of the Darboux-Halphen equation. We give an algebraic condition  for a  paramaterized linear differential equation to  be projectively isomonodromic, in terms of the derived group of its parameterized Picard-Vessiot group.
\end{abstract}

\maketitle
\section {Introduction} 
Classically, monodromy preserving deformations of Fuchsian systems have been investigated   by many authors who described them in terms of the Schlesinger equation and its links to Painlev\'e equations. In \cite{Landesman}, Landesmann developed a new Galois theory for  parameterized differential equations. A special case was developed   in \cite{CaSi} where the authors consider parameterized {\it linear}  differential equations and discuss various properties of the  parameterized Picard-Vessiot group, the PPV-groups for short. This is a linear differential algebraic group in the sense of Cassidy~\cite{cassidy1}. As is well known, the differential Galois group of a system with regular singularities is, as a linear algebraic group, Zariski topologically generated by the monodromy matrices  with respect to a fundamental solution. Cassidy and Singer have shown that a parameterized family of such systems is isomonodromic if and only if its PPV-group is conjugate to a (constant) linear algebraic group.

 Analogous to the Schlesinger and Painlev\'e equations' relation to isomonodromic deformations of Fuchsian systems,  the Darboux-Halphen~V equation accounts for a special type of monodromy evolving deformation of Fuchsian systems, as was shown by Chakravarty and Ablowitz in~\cite{ChAb}. After reviewing the notion of isomonodromy in Section~\ref{classIso}, we follow and detail, in Section~\ref{monoevo}, the  description by Ohyama \cite{ohyama} of the Darboux-Halphen  system  in a way that will illustrate the general notion of  projective isomonodromy that is introduced in Section~\ref{projmono} for general parameterized (not necessarily Fuchsian) 
 differential systems with analytic coefficients. For Fuchsian systems, our de\-fi\-ni\-tion naturally extends the classical notion of isomonodromy (not necessarily of the Schlesinger type)  given for instance in \cite{Bol_iso_def}.   For these systems, we  show in Section~\ref{isovproj}  that projective isomonodromy  is indeed the type of monodromy evolving deformation introduced by Ohyama in (\cite{ohyama},~Section 4) and we characterize it by  a  condition on the residue matrices. In Section~\ref{galois}, we consider general parameterized systems with regular singularities that are not necessarily Fuchsian,  and we characterize  pro\-jec\-tive isomonodromy by the purely algebraic condition that the derived group $(G,G)$ of the PPV-group $G$ be conjugate to a constant linear algebraic group when the given system is absolutely irreducible. 
 
We wish to thank Stephane Malek for making us aware of \cite{ohyama}.

\section{Classical isomonodromy} \label{classIso}

In many of  the classical studies of isomonodromic deformations,    only parameterized  {\it Fuchsian} systems  are  considered. Furthermore, these systems are assumed to be parameterized in a very special way, that is, the systems are written as 

\begin{equation}\label{isocl}
\frac{dY}{dx}=\sum_{i=1}^{m}\frac{A_i(a)}{x-a_i}, \ \ \ \sum_{i=1}^{m}{A_i(a)}=0
\end{equation}
where the  $n\times n$ matrices  $A_i(a)$ depend holomorphically on the multi-parameter $a=(a_1,\ldots,a_n)$  in  some open  polydisk~$D(a^0)$ and the condition on the residue matrices guarantees, for simplicity,  that $\infty$ is not singular. The polydisk $D(a^0)=
D_1\times\ldots\times D_m$ has center at the initial location $a^0=(a^0_1,\ldots,a^0_m)\in \CX^m$ of the poles, with  $D_i\subset \CX$ a disk with center  $a^0_i$ and $D_i\cap D_j\ne \emptyset$ for all $i\ne j$.  Let  $x_0\in\calD=\pp\setminus\bigcup_i D_i$.  

 For fixed $a\in D(a^0)$ and local fundamental solution $Y_a$  of (\ref{isocl}) at $x_0$,  analytic continuation along a loop $\gamma$  from $x_0$ in $\calD_a=\pp\setminus\{a_1,\ldots,a_m \}$ yields a solution $Y_a^{\gamma}$. The {\it monodromy representation} with respect to $Y_a$ is 
\begin{eqnarray}\label{chi}
\chi_a : \pi_1(\calD_a;x_0)\rightarrow \GL_n(\CX)
\end{eqnarray}
defined  by
\[Y_a^{\gamma}=Y_a.\chi_a(\gamma),  \]
for all $[\gamma]\in \pi_1(\calD_a;x_0)$.

\begin{defin} 
Equation \eqref{isocl} is  {\it isomonodromic}, or an  {\it isomonodromic deformation},  if for each $a\in D(a^0)$  there is a matrix  $C(a)\in \GL_n(\CX)$ such that
$$\chi_a=C(a)\ \chi_{a^0}\ C(a)^{-1}.$$
\end{defin}

\noindent Bolibrukh (\cite{Bol_iso_def}, \cite{Bol_cfl}) has characterized isomonodromic deformations as follows. 
\begin{thm}[Bolibrukh]\label{bol}  
Equation \eqref{isocl} is  isomonodromic if and only if the following equivalent conditions hold.
\begin{enumerate}
\item There is a differential $1$-form $\omega$ on 
$\left(\pp\times D(a^0)\right)\setminus \bigcup_{i=1}^m
\{(x,a)\ | \ x-a_i=0 \}$ such that
\begin{itemize}
\item for each fixed $a\in D(a^0),$
$$\omega=\sum_{i=1}^{m}\frac{A_i(a)}{x-a_i} dx$$
\item $d\omega=\omega\wedge\omega.$
\end{itemize}
\item For each  $a\in D(a^0)$ there is a fundamental solution~$Y_a$ of (\ref{isocl}) such that~$Y_a(x)$ is analytic in $x$ and $a$,  and  the corresponding monodromy representation $\chi_a$ does not depend on $a$, that is, $\chi_a=\chi_{a^0}$.
\end{enumerate} 
\end{thm}
\noindent A special type of isomonodromic deformation is given by the  Schlesinger differential form
\begin{equation} \omega_s=\sum_{i=1}^m \frac{A_i(a)}{x-a_i}d(x-a_i) 
\end{equation}
whose integrability condition is known as the {\it Schlesinger equation}
\begin{equation} \label{schl} dA_i(a)=-\sum_{j=1, j\ne i}^m \frac{[A_i(a),A_j(a)]}{a_i-a_j}d(a_i-a_j), \ \ i=1,\ldots m.\end{equation}

\noindent Bolibrukh gave examples \cite{Bol_iso_def} of isomonodromic deformations that are not of the Schlesinger type and he described the general differential forms  that occur in  Theorem \ref{bol}.

 In the special case of  order two Fuchsian systems with four singularities one can, generically,  reduce each system to an  order two linear scalar differential equation satisfied by the first component of the dependent variable $Y$,  namely a Fuchsian scalar equation with  an additional  apparent singularity $\lambda$.  It is well known that the Schlesinger isomonodromy condition then translates into a non-linear equation of Painlev\'e VI  type satisfied by $\lambda$.  For  basic results about  Painlev\'e equations and isomonodromic deformations, we refer to \cite{japanese} and \cite{AbCla} .

\section{An example of a monodromy evolving deformation}\label{monoevo} In  \cite{ChAb},  Chakravarty and Ablowitz  describe  the Darboux-Halphen system
\begin{eqnarray}\label{DaH}
\left\{ \begin{array}{cccccccc}
\omega_1' &= &&\omega_2\omega_3 &-&\omega_1(\omega_2+\omega_3)&+& \phi^2 \\
\omega_2' &=&& \omega_3\omega_1 &-&\omega_2(\omega_3+\omega_1)&+ &\theta^2 \\
\omega_3' &=&& \omega_1\omega_2 &-&\omega_3(\omega_1+\omega_2)&- &\theta\phi \\
\phi'         &=&&\omega_1(\theta - \phi) &-&\omega_3(\theta + \phi) &&\\
\theta'      &=& - &\omega_2(\theta- \phi) &-&\omega_3(\theta + \phi) &&\\
\end{array}
\right. 
\end{eqnarray}
as a prototype of a class of non-linear systems arising as the integrability conditions of an associated Lax pair in the same way as the Painlev\'e and Schlesinger equations do.
This system occurs in the Bianchi IX cosmological models and  arises  from a special reduction of the self-dual Yang-Mills (SDYM) equation ({\it cf}.~\cite{AbCla}, \cite{ChAb}, \cite{ChAb2}, \cite{ohyama}). It is also related to the Chazy and Painlev\'e VI equations (see \cite{AbCla} for a complete study of such equations and  reductions of  the SDYM equation). We will  review and detail Ohyama's study \cite{ohyama} of  this equation  and refer to \eqref{DaH} as the Darboux-Halphen V Equation or DH-V for short. 

Originally ({\it cf}.~\cite{ChAb}) the DH-V system  with the special condition $\theta=\phi=0$, called the Halphen II equation (H II),  arose from a geometrical problem studied by Darboux. In 1878 Darboux  obtained this equation as the integrability condition for the existence in  Euclidean space of a one-parameter family of surfaces  orthogonal to two arbitrary given independent families of parallel surfaces (such a family is necessarily quadratic and ruled).  Halphen solved this system in 1881. 

 Ohyama (\cite{ohyama}, \cite{ohyama2}) shows how H II is in the generic case equivalent to  
\begin{equation}\label{Halphen} x_i'=Q(x_i), \ \  i=1,2,3
\end{equation}
where 
\begin{eqnarray}\label{quadr}
Q(x)=x^2+a(x_1-x_2)^2+b(x_2-x_3)^2+c(x_3-x_1)^2
\end{eqnarray}
with constants $a,b,c$  such that $a+b=c+b=-{1}/{4}$
(all derivatives are with respect to the complex variable $t$). 

\ As pointed out in \cite{ohyama},  these equations do not satisfy the Painlev\'e property (on their movable singularities) and may therefore not be expected  to be monodromy-preserving conditions. Nevertheless Chakravarty and Ablowitz in \cite{ChAb}, and Ohyama in  \cite{ohyama},  showed how these  non-linear equations actually express a special type of  monodromy evolving deformation, in the same way as the Schlesinger  and Painlev\'e VI equations rule the isomonodromic deformations of the Schlesinger type. 

 Using the connection relating the self-dual Yang-Mills equation and the conformally self-dual Bianchi equations, these authors showed that DH-V, and hence Equation \eqref{Halphen},  actually is the compatibility condition of   a Lax pair
\begin{eqnarray}\label{eqn1}
 \frac{\partial Y}{\partial x} & = & \left( {\frac{\mu}{P}}I+\sum_{i=1}^3{\frac{\lambda_iS}{x-x_i}} \right)Y
\end{eqnarray}
 \begin{eqnarray}\label{eqn2}
\frac{\partial Y}{\partial t} & = & \left( \nu I+\sum_{i=1}^3 \lambda_ix_iS\right)Y-Q(x) {\frac{\partial Y}{\partial x} }
\end{eqnarray}
of $2\times 2$ matrix equations, where $x_1,x_2,x_3$ depend on $t$ and $P(x)=(x-x_1)(x-x_2)(x-x_3)$,  and $S$ is a traceless constant matrix   and where $\mu$ and the $\lambda_i$ are constants  with $\mu\ne 0$,  $\lambda_1+\lambda_2+\lambda_3=0$,  and $\nu(x,t)$ satisfies the  auxiliary equation
\begin{eqnarray}\label{eqn7a}
\frac{\partial \nu}{\partial x}=-{\frac{x+x_1+x_2+x_3}{P}}\mu.
\end{eqnarray}

\noindent Assume that the Lax pair is integrable, {\it i.e.} that Equation \eqref{Halphen} is satisfied by the  $x_i$. Equation (\ref{eqn1}) Êis for fixed $t$ a Fuchsian system with three singular points $x_1,x_2,x_3$, and
Equation (\ref{eqn7a}) implies that $\nu$  is not a rational function of $x$.  Since  the coefficients of (\ref{eqn2}) are not rational,  the Lax pair ((\ref{eqn1}), (\ref{eqn2})) does not describe an isomonodromic deformation ({\it cf.} \cite{Sibuya}, Remark A.5.2.5).

 Let us  fix $t_0\in \CX$, and  open disjoint disks $D_i$ with center at $x_i(t_0)$, $i=1,2,3$.  Let   { $U(t_0)$} be a neighborhood of $t_0$ in $\CX$ such that $x_i(t)\in D_i$ for each $i$ and all~$t\in U(t_0)$, and let $x_0\in \CX$ be a fixed base-point,  $x_0\notin \bigcup_i D_i$. 

Let $Y(t,x)$, for $t\in { U(t_0)}$,  denote a  fundamental solution, in a neighborhood of $x_0$,  of the Lax pair (\eqref{eqn1}, \eqref{eqn2}).  It is therefore analytic in both $t$ and $x$. For fixed $t\in { U(t_0)}$, we can write an analytic continuation of the fundamental solution~$Y(t,x)$ to a punctured neighborhood of $x_i$ as  
 \begin{eqnarray}\label{eqn4}
{Y(t,x)=Y_i(t,x-x_i(t)).(x-x_i(t))^{L_i(t)}}
\end{eqnarray}
where $Y_i(t, x-x_i(t))$  is single-valued, and the matrix $L_i(t)$ does not depend on $x$.  Note that $Y_i(t,x-x_i(t))$ is analytic in $t$ and $x$ and $L_i(t)$ is analytic in $t$. Indeed, for  fixed  $t\in { U(t_0)}$,   analytic continuation  of $Y$ along an elementary loop around $x_i(t)$ yields a fundamental solution ${\tilde Y}(t,x)$ of \eqref{eqn1} which is again analytic in both $t$ and $x$,  by the theorem about analytic dependence on  initial conditions ({\it cf.} \cite{cartan}). The monodromy matrix  $M_i(t)$ is  therefore analytic in $t$, as well as $L_i=(1/2\pi i)\log M_i(t)$, and hence  $Y_i(t,x-x_i(t))=Y(t,x)\ (x-x_i(t))^{-L_i(t)}$ is analytic in $t$ and $x$ in $(U(t_0)\times D_i)\setminus \{(t,x)\ | \ x-x_i(t)=0 \}$.

 \begin{prop}\label{prop0}
With notation as above, let $M_i(t)$ for any fixed $t\in U(t_0)$ denote the monodromy matrix of (\ref{eqn1}) with respect to $Y$, defined by analytic continuation along  an elementary loop around $x_i(t)$. Then
\begin{eqnarray}
M_i(t)=c_i(t)G_i
\end{eqnarray}
where $G_i$ is a constant matrix and where~$c_i(t)=e^{-2\pi\mu\sqrt{-1}\int_{t_0}^t \alpha_i(t)dt}$ and the $\alpha_i$ are the residues of
\begin{eqnarray}\label{projfact}
 {\frac{x+x_1+x_2+x_3}{P}}=\sum_{i=1}^{3}\frac{\alpha_i}{x-x_i}.
\end{eqnarray}
\end{prop} 

\begin{proof}
Let us show that
\begin{eqnarray}\label{exp}
 \frac{d L_i}{d t}=-\alpha_i \mu I
 \end{eqnarray}
 where $\alpha_i$ is the $x_i$-residue of
$(x+x_1+x_2+x_3)/P$, that is, 
 \[ \alpha_i=\frac{x_i+\sigma}{\prod_{j\ne i}(x_i-x_j)}\] 
 with~$\sigma=x_1+x_2+x_3$. Note that  $Y_i$  is  a function (the sum of a series) of the local	 coordinate $x-x_i$ . We  have
 \[ \frac{\d Y}{\d t}=- Q(x_i) \frac{\d Y_i}{\d x}\cdot (x-x_i)^{L_i} + Y_i \cdot  (x-x_i)^{L_i}\left(\log (x-x_i) \frac{d L_i}{d t}-Q(x_i)\frac{ L_i}{x-x_i}\right)\] 
 \[=-Q(x_i)  \frac{\d Y}{\d x}       + Y\cdot \log (x-x_i) \frac{d L_i}{d t}. \]
 
If we compare with  Equation \eqref{eqn2} of the Lax pair 
we get 
 \[ -Q(x_i)\frac{\d Y}{\d x}+ Y\cdot \log (x-x_i)\frac{d L_i}{d t}=-Q(x)\frac{\d Y}{\d x} +\left(\nu I +\sum_{i=1}^{3}c_ix_iS \right)Y.\]
From Equation  (\ref{eqn7a}) we have that
\[ \nu=\mu\log \prod_{i=1}^3(x-x_i)^{-\alpha_i}+\phi(t)\]
 for some function $\phi(t)$,  and hence as $x$ tends to $x_i$ for fixed $t$ (simplifying and then comparing the leading terms on each side) we get that 
 \[ \log (x-x_i)\frac{d L_i}{d t}\sim -\alpha_i \mu \log (x-x_i)I,\]
that is,
 \[ \frac{d L_i}{d t}=-\alpha_i \mu I .\]
The monodromy matrix  of (\ref{eqn1}) with respect to $x_0$ and $Y$ around  $x_i$ is $M_i=e^{2\pi i L_i}$, which in view of  (\ref{exp}) is of the form
\[ M_i(t)=c_i(t)G_i \]
where $G_i$ is the initial monodromy matrix around $x_i(t_0)$, and $$c_i(t)=e^{-2\pi\mu\sqrt{-1}\int_{t_0}^t \alpha_i(t)dt}.$$ 
  \end{proof}
 This is an example of what we will call projectively  isomonodromic deformations, and study from an algebraic point of view.   

\section{Projective isomonodromy }\label{projmono}
Let  $\calD$ be an open connected subset of  $\PX^1(\CX)$, $\calP$ be an open connected subset of  $\CX^r$, and $x_0 \in \calD$.  Assume that $\pi_1(\calD,x_0)$ is finitely generated by $\gamma_1, \ldots ,\gamma_m$.  Let $A(x,\tbar) \in \gl_n(\calO)$, where $\calO$ denotes the ring of  $n \times n$ matrices whose entries  are functions analytic on $\calD \times \calP$.  We will consider the behavior of  solutions of the differential equation
\begin{eqnarray}\label{eqn5}
\frac{dY}{dx} &= &A(x,\tbar)Y.
\end{eqnarray}
In the following we let $\Scal_n$ be the group of nonzero $n\times n$ scalar matrices.
\begin{defin} Equation (\ref{eqn5}) is {\it projectively isomonodromic} if there exist $m$ analytic functions $c_i:\calP \rightarrow \Scal_n(\CX)$  and fixed matrices $G_1, \ldots , G_m \in \GL_n(\CX)$ such that for each $\tbar \in \calP$ there is a local solution $Y_\tbar(x)$ of (\ref{eqn5}) at $x_0$ such that analytic continuation of $Y_\tbar(x)$ along $\gamma_i$ yields $Y_\tbar(x) \cdot G_i c_i(\tbar)$, for each $i$.
\end{defin}

 Let  $\bar{Y}(x, \tbar)$   be any solution of (\ref{eqn5}) analytic in $\calD_0\times\calP$, where $\calD_0$ is a neighborhood of $x_0$ in $\calD$ and let $M_i(\tbar)$ denote the monodromy matrix corresponding to analytic continuation of this solution around $\gamma_i$. Note that $M_i(\tbar)$ depends analytically on $\tbar$. If (\ref{eqn5}) is projectively isomonodromic then there exists a function  $C:\calP \rightarrow \GL_n(\CX)$  such that \[M_i(\tbar) = C(\tbar)^{-1} G_i c_i(\tbar)C(\tbar)\] for all $t\in \calP$.
Since there may be many ways of selecting $C(\tbar)$, this function need not depend analytically on $\tbar$.  However, we will show that one can find a function ${C}(\tbar)$  satisfying the above {\it and  analytic in $\tbar$}. This fact can be deduced  easily from the following result of Andrey Bolibruch whose proof is contained in the proof of Proposition 1 of \cite{Bol_iso_def}.
\begin{prop} \label{prop2} For each $ i=1, \ldots , m$, let $H_i: \calP \rightarrow \GL_n(\CX)$   be analytic on $\calP$ and let $G_i \in  \GL_n(\CX)$.  Assume that there is a function $C:\calP \rightarrow \GL_n(\CX)$ such that 
\[H_i(t)=C(t)^{-1}G_iC(t)\]
 for all $t \in \calP$ and  $i=1,\ldots,m$.  Then there exists an analytic function  $C$ with the same property.  
\end{prop}
We can now prove the following
\begin{prop}\label{prop3} If (\ref{eqn5}) is projectively isomonodromic, then there exists a solution $Y(x, \tbar)$    of (\ref{eqn5}) analytic in $\calD_0\times\calP$, where $\calD_0$ is a neighborhood of $x_0$ in $\calD$ such that for all $\tbar \in \calP$ the monodromy matrix of $Y(x, \tbar)$ along $\gamma$ \underline{is} $ G_i  c_i(t)$.
\end{prop}
\begin{proof} Let  $\bar{Y}(x, \tbar)$   be any solution of (\ref{eqn5}) analytic in $\calD_0\times\calP$, where $\calD_0$ is a neighborhood of $x_0$ in $\calD$ and let $M_i(\tbar)$ denote the monodromy matrix corresponding to analytic continuation of this solution around $\gamma_i$.  
Since (\ref{eqn5}) is projectively isomonodromic, there is a function $C:\calP \rightarrow \GL_n(\CX)$ such that $M_i(t)=C(t)^{-1}G_ic_i(t)C(t)$ for all $t\in \calP$. Applying Proposition~\ref{prop2} to $H_i(t)=M_i(t){ c_i(t)^{-1}}$ and $G_i$, we may assume that  $C(t)$ is analytic and thus  $Y(x,\tbar) = \bar{Y}(x,\tbar) {C}(\tbar)$ satisfies the conclusion of this Proposition.\end{proof}
\section{Isomonodromy versus projective isomonodromy}\label{isovproj}

\noindent  We now turn to the relation between Fuchsian isomonodromic equations and Fuchsian projectively isomonodromic equations. Consider the equation
 \begin{eqnarray}\label{eqn6}
 \frac{dY}{dx} & = & \sum_{i=1}^m \frac{A_i(\tbar)}{x-x_i(\tbar)}Y 
 \end{eqnarray}
together with
 \begin{enumerate}
 \item $\calP$, a  simply connected open subset of  $\CX^r$ and 
 \item $\calD$, an open subset  in $\pp$ and $x_0 \in \calD$
 such that     
 \item the functions $A_i: \calP \rightarrow \gl_n(\CX)$ and the $x_i: \calP \rightarrow \CX$ are analytic functions,
 \item $\pp\backslash\calD$ is the union of $m$ disjoint closed disks $D_i$ and 
 \item for $\tbar\in \calP$ we have $x_i(\tbar) \in D_i$.
 \end{enumerate}
 Let $x_0 \in \calD$ and $\gamma_i$,  $i = 1, \ldots , m,$ be the obvious  loops  generating $\pi_1(\calD, x_0)$. We then have that Equation (\ref{eqn6}) is analytic in $\calD \times \calP$ and we can speak of monodromy matrices $M_i(\tbar)$ corresponding to analytic continuation of a fundamental solution matrix along $\gamma_i$.  We can now state
 \begin{prop}\label{descr} Let $\calD$ and $\calP$ be as above.  Equation (\ref{eqn6}) is projectively  isomonodromic if and only if for each $ i=1, \ldots , m$, there exist functions $b_i:\calP \rightarrow \Scal_n(\CX)$ and $B_i:\calP \rightarrow \gl_n(\CX)$, analytic on $\calP$ such that 
 \begin{enumerate}
 \item $A_i= B_i+ b_i $ for $i = 1, \ldots , m$ and
 \item \begin{eqnarray}\label{eqn7}
 \frac{dY}{dx} & = & \left(\sum_{i=1}^m\frac{B_i(\tbar)}{x-x_i(\tbar)}\right)Y
 \end{eqnarray}
 is isomonodromic.
 \end{enumerate}
 \end{prop}
 \begin{proof} Assume that Equation (\ref{eqn6}) is projectively isomonodromic and let $Y(x,\tbar), G_i$ and $c_i$ be as in the conclusion of Proposition~\ref{prop3}.  Since $\calP$ is simply connected and the $c_i(\tbar)$ are nonzero for all $t$, there exist analytic $b_i:\calP \rightarrow \Scal_n(\CX)$ such that $e^{2\pi\sqrt{-1}b_i} = c_i$. Let 
 \[Z(x,\tbar) = Y(x,\tbar) \prod_{i=1}^m(x-x_i(\tbar))^{-b_i(\tbar)}\]
 One sees that the monodromy of $Z$ along $\gamma_i$ is given by $G_i$ and so is independent of $\tbar$.  Therefore, letting $B_i =  A_i- b_i$, we have that 
 
 \item \begin{eqnarray*}
 \frac{dY}{dx} & = & \left(\sum_{i=1}^m \frac{B_i(\tbar)}{x-x_i(t)}\right)Y
 \end{eqnarray*}
 is isomonodromic.

 Now assume that $A_i,B_i,b_i$ are as in items (1)~and (2)~of the proposition and that Equation (\ref{eqn7}) is isomonodromic.  If $Y(x,\tbar)$ is a local solution of (\ref{eqn7}) with constant monodromy matrices $G_i$ along $\gamma_i$, then 
$Z(x,\tbar) = Y(x,\tbar) \prod_{i=1}^m (x-x_i(\tbar))^{b_i(\tbar)}$  
 will have monodromy $G_i c_i(t)$  along $\gamma_i$, with
 $c_i(\tbar) = e^{2\pi\sqrt{-1}b_i(\tbar)}\ $.  Thus  Equation~(\ref{eqn6}) is projectively isomonodromic.
 \end{proof}
 Proposition \ref{descr} applies  to the DH-$\mathrm V$ example since we can rewrite Equation (\ref{eqn1}) of the Lax pair as
 \begin{eqnarray*} \frac{\partial Y}{\partial x} = \left(\sum_{i=1}^3 \frac{A_i(t)}{(x-x_i)}\right) Y
\end{eqnarray*}
where  $
 A_i =  B_i + b_i$, with 
 \begin{eqnarray*}
B_i & = & \lambda_iS\\
b_i&=& {\frac{\mu I_n}{\prod_{j\neq i}(x_i-x_j)}}.
\end{eqnarray*}
An easy computation shows that since  $x'_i-x'_j=Q(x_i)-Q(x_j)=x_i^2-x_j^2$ for all $i,j$, we have
\[ b'_i=\frac{db_i}{dt}= -{\frac{x_i+\sigma }{\prod_{j\neq i}(x_i-x_j)}}\mu==-\alpha_i \mu.\]
and we recover the result of Proposition~\ref{prop0},  that the monodromy of this equation is  evolving `projectively' and  equal to  \[ M_i(t)=e^{2\pi\sqrt{-1}b_i(\tbar)}G_i= e^{-2\pi\mu\sqrt{-1}\int_{t_0}^t \alpha_i(t)dt}G_i.\]

 \section{Parameterized differential Galois groups}\label{galois} 
 In this section we examine the parameterized differential Galois groups of projectively isomonodromic equations.   Parameterized differential Galois groups Ê({\it cf.} \cite{CaSi}, \cite{Landesman}) generalize the concept of differential Galois groups of the classical Picard-Vessiot theory and we begin this section by briefly describing the underlying theory. 
 
 Let 
 \begin{eqnarray}\label{pveqn}\frac{d Y}{d x} &=& A(x) Y\end{eqnarray}
 be a differential equation where $A(x)$ is an $n\times n$ matrix with entries in $\CX(x)$.  The usual existence theorems for differential equations imply that if $x = x_0$ is a point in $\CX$ such that the entries of $A(x)$ are analytic at $x_0$, then there exists a nonsingular matrix $Z = (z_{i,j})$ of functions analytic in a neighborhood of $x_0$ such that $\frac{d Z}{d x} = A(x)Z$. Note that the field $K = \CX(z_{1,1}, \ldots , z_{n,n})$ is closed with respect to taking the derivation $\frac{d}{d x}$ and  this is an example of a {\it Picard-Vessiot extension}\footnote{Picard-Vessiot extensions and the related Picard-Vessiot theory is developed in a fuller generality in \cite{DAAG} and \cite{PuSi2003} but we shall restrict ourselves to the above context to be concrete.}.  The set of field-theoretic isomorphisms of $K$ that leave $\CX(x)$ elementwise fixed and commute with  $\frac{d}{d x}$ forms a group $G$ called the {\it Picard-Vessiot group} or {\it differential Galois group} of  (\ref{pveqn}). One can show that for any $\sigma \in G$, there exists a matrix $M_\sigma \in \GL_n(\CX)$ such that $\sigma(Z) =(\sigma(z_{i,j})) = ZM_\sigma$. The map $\sigma \mapsto M_\sigma$ is an isomorphism whose image is furthermore a {\it linear algebraic group}, that is, a group of invertible matrices whose entries satisfy some fixed set of polynomial equations in $n^2$ variables. There is a well developed Galois theory for these groups that describes a correspondence between certain subgroups of $G$ and  certain subfields of $K$ as well as associates properties of the equation  (\ref{pveqn}) with properties of the group $G$.  The elements of the monodromy group of (\ref{pveqn}) may be identified with elements of this group and,  when (\ref{pveqn}) has only regular singular points, it is known that $G$ is the smallest linear algebraic group containing these elements ({\it cf.} \cite{PuSi2003}, Theorem 5.8). Further facts about this Galois theory can be found in \cite{DAAG} and \cite{PuSi2003}. 
 
 Now let 
 \begin{eqnarray}\label{ppveqn}
 \frac{d Y}{d x} & = & A(x,\tbar)Y\end{eqnarray}
 be a parameterized system of linear differential equations where $A(x,\tbar)$ is an $n\times n$ matrix whose entries are rational functions of $x$ with coefficients that are functions of $\tbar = (t_1, \ldots, t_r)$, analytic in some domain in $\CX^r$. A differential Galois theory for such equations was developed in \cite{CaSi} and in greater generality in \cite{Landesman}. Let $k_0$ be a suitably large field\footnote{To be precise, we need $k_0$ to be  {\it differentially closed} with respect to $\Pi$, that is, any system of polynomial differential equations in arbitrary unknowns having a solution in an extension field already has a solution in $k_0$. See \cite{CaSi} for a discussion of differentially closed fields in the context of this Galois theory.} containing $\CX(t_1, \ldots, t_r)$ and the functions of $\tbar$ appearing as coefficients in the entries of $A$ and such that $k_0$ is closed under the derivations $\Pi = \{\d_1, \ldots, \d_r\}$ where each $\d_i$ restricts to $\frac{\d}{\d t_i}$ on $\CX(t_1, \ldots, t_r)$ and  the intersection of the kernels of the $\d_i$ is $\CX$. As before, existence theorems for solutions of differential equations guarantee the existence of a nonsingular matrix $Z(x,\tbar) = (z_{i,j}(x,\tbar))$ of functions, analytic in some suitable domain in $\CX \times \CX^r$, such that $\frac{d Z}{d x} = AZ$.  
 We will let $k = k_0(x)$ be the differential field with derivations $\Delta = \{\d_x, \d_1, \ldots , \d_r\}$ where $\d_x(x) = 1, \ \d_x(z) = 0 \mbox{ for all $z \in k_0$ and the } \d_i \mbox{ extend the previous } \d_i \mbox{ with } \d_i(x) = 0.$ Finally we will denote by $K$ the smallest field containing $k$ and the $z_{i,j}$ that is closed under the derivations of $\Delta$. This field is called the {\it parameterized Picard-Vessiot field} or  PPV-field  of  (\ref{ppveqn}).  The set of field-theoretic automorphisms of $K$ that leave $k$ elementwise  fixed and commute with the elements of  $\Delta$ forms a group $G$ called the {\it  parameterized Picard-Vessiot group } (PPV-group) or {\it parameterized differential Galois group} of  (\ref{ppveqn}). One can show that for any $\sigma \in G$, there exists a matrix $M_\sigma \in \GL_n(k_0)$ such that $\sigma(Z) =(\sigma(z_{i,j})) = ZM_\sigma$. Note that $\d_x$ applied to an entry of such an $ M_\sigma$ is  $0$ since these entries are elements of $k_0$ but that such an entry need not be constant with respect to the elements of $\Pi$.  One may think of these entries as functions of $\tbar$. In \cite{CaSi}, the authors show that the map $\sigma \mapsto M_\sigma$ is an isomorphism whose image is furthermore a {\it linear differential algebraic group}, that is, a group of invertible matrices whose entries satisfy some fixed set of polynomial {\it differential}  equations (with respect to the derivations $\Pi = \{\d_1, \ldots, \d_r\}$) in $n^2$ variables. We say that a set $X \subset \GL_n(k_0)$ is {\it Kolchin-closed} if it is the zero set of such a set of polynomial differential equations.  One can show that the Kolchin-closed sets form the closed sets of a topology, called the {\it Kolchin topology} on $\GL_n(k_0)$  ({\it cf.} \cite{cassidy1, cassidy6, CaSi, kolchin_groups}). 

The following result shows how the PPV-group can be used to characterize isomonodromy. As in Section 4, let $\calP$ be a simply connected subset of $\CX^r$ and $\calD$ an open subset of $\pp$ with $x_0 \in \calD$.  We assume that $A(x,\tbar)$ in Equation (\ref{ppveqn})  is analytic in $\calD\times \calP$.  Assume that $\pp\backslash \calD$ is the union of $m$ disjoint disks $D_i$ and that for each $\tbar \in \calP$, Equation (\ref{ppveqn})  has a unique singular point in each $D_i$ and that this singular point is a {\it parameterized regular singular point} in the sense of (\cite{RegSing}, Definition 2.3).  Note that by  (\cite{RegSing}, Corollary 2.6) this in particular implies that the singularity is regular singular for each fixed $t$,  in the usual sense.  Let $\gamma_i, i = 1, \ldots , m$ be the obvious  loops  generating $\pi_1(\calD, x_0)$. We then have that Equation (\ref{ppveqn}) is analytic in $\calD \times \calP$ and we can speak of (parameterized) monodromy matrices $M_i(\tbar)$ corresponding to analytic continuation,  for each fixed  $t$,  of a fundamental solution matrix along $\gamma_i$. 

 \begin{prop}({{\it cf.} \cite{CaSi}, Proposition 5.4}) Assume that $\calD$, $\calP$ and Equation (\ref{ppveqn}) are as above. Then this equation is  isomonodromic in $\calD\times \calP'$ for some subset $\calP' \subset \calP$ if and only if the PPV-group $G$ of this equation over $k$ is conjugate to $G_1(\CX)$ for some linear algebraic group $G_1$ defined over $\CX$. 
 \end{prop}

 \section {An algebraic condition for projective isomonodromy}
 We now relate the property of projective isomonodromy to properties of the PPV-group. We still assume Equation \eqref{ppveqn} has (parameterized) regular singularities only, in the sense of  (\cite{RegSing}, Definition 2.3).

 \begin{prop} \label{prop3.1}  Let $k,K, A,$ and $G$ be as above.  Equation (\ref{ppveqn}) is projectively isomonodromic if and only if its PPV-group $G$ is conjugate to a subgroup of $$ \GL_n(\CX) \cdot \Scal_n(k_0)\subset \GL_n(k_0).$$
\end{prop}
\begin{proof} 
Assume Equation (\ref{ppveqn}) is projectively isomonodromic and  let $x_i(t), \ i = 1, \ldots , m$ be the singular points of Equation \eqref{ppveqn}
\begin{eqnarray*}\label{ppveqn2}
 \frac{d Y}{d x} & = & A(x,t)Y\end{eqnarray*}
and let $Y(x,t)$,  $G_i$,  $c_i$ be as in the conclusion of Proposition 6. Since $\calP$ is simply connected and the $c_i(t)$ are nonzero, there exist analytic $ b_i:\calP \rightarrow  \Scal_n\CX$ such that $ e^{2\pi\sqrt{-1}b_i} = c_i$ for each $i$ (in fact, we can select $b_i(t) \in  \Scal(k_0)$ since  $k_0$ is differentially closed). Consider the system of differential equations 
\begin{eqnarray}\label{ppveqn2}
 \frac{d Z}{d x} & = & (A(x,t)-\sum_{i=1}^m \frac{b_i(t)}{(x-x_i(t))})Z = B(x,t) Z\nonumber\\ &&\\
 \frac{du}{d x} & = & (\sum_{i=1}^m \frac{b_i(t)}{(x-x_i(t))})u=b(t) u\nonumber\end{eqnarray}

 We will consider this as a differential equation with coefficient matrix
 \[\left(\begin{array}{cc} B & 0\\0&b\end{array}\right) \in  \gl_{2n}(k).\]
 Let $E$ be the PPV-extension of $k$ for the system (\ref{ppveqn2}).  Note that the first equation of this system is isomonodromic, hence  has a PPV-group that is conjugate in $\GL_n(k_0)$ to a constant group over $\CX$. Therefore the PPV-group  of (\ref{ppveqn2}) over $k$ is conjugate in  $\GL_{ 2n}(k_0)$ to a subgroup of $\GL_n(\CX) \times \Scal_n(k_0)\subset  \GL_{2n}(k_0)$. 
 It is easy to see that $Y\in \GL_n(E)$ is a fundamental solution of Equation \eqref{ppveqn} if and only if~$Y=Zu$ where 
\[ \left(\begin{array}{cc} Z & 0\\0&u\end{array}\right)\]
 is a fundamental solution of \eqref{ppveqn2}. Let us choose $Z$ such  that the corresponding representation $H$ of the PPV-group is a subgroup of $\GL_n(\CX) \times \Scal_n(k_0)$ and let $Y=Zu$ be a  fundamental solution of \eqref{ppveqn}. We clearly have $K \subset E$, where $K$ is the PPV-extension of $k$ for \eqref{ppveqn}.  The action of $H$ on $E$ gives 
  \[ (\sigma, \tau)(Y) = \sigma(Z)\tau(u)\] 
 for all $(\sigma, \tau)\in H$, and it induces a homomorphism $\Phi$ of $H$ onto the PPV-group of $K$ over $k$.  In terms of matrices, $\Phi$ is given by $\Phi(\sigma,\tau)  =  \sigma\cdot \tau $ and its image is in $\GL_n(\CX)\cdot \Scal_n(k_o)$.

The converse is clear since  the (parameterized) monodromy matrices $M_i(t)$ all belong to the PPV-group (see Theorem 3.5 of \cite{RegSing}). 
\end{proof}
One easy consequence of Propostion~\ref{prop3.1} is 
\begin{cor}\label{cor3.1} Let $k,K, A,$ and $G$ be as above. If Equation (\ref{ppveqn}) is projectively isomonodromic then  the commutator subgroup $(G,G)$ is conjugate to a subgroup of $\GL_n(\CX)$. \end{cor}
This corollary yields a simple test to show that (\ref{ppveqn}) is not projectively isomonodromic: If the eigenvalues of the commutators of the monodromy matrices (with respect to any fundamental solution matrix) are not constant, then (\ref{ppveqn})  is not projectively isomonodromic. In particular,  if the determinant or trace of any of these matrices is not constant then (\ref{ppveqn})  is not projectively isomonodromic. The converse of the corollary is not true in general (see Remark~\ref{remark3.1} below) but it is true if Equation (\ref{ppveqn}) is absolutely irreducible, that is, when  (\ref{ppveqn}) does not factor over $\overline{k}$, the algebraic closure of $k$. Before we prove this, we will discuss some group theoretic facts.

 In the following, we say that a subgroup $H \subset \GL_n(k_0)$ is {\it irreducible} if the only $H$-invariant subspaces of $k_0^n$ are $\{0\}$ and $k_0^n$. A differential equation with coefficients in $k$ is {\it absolutely irreducible} if it is irreducible over any finite extension of $k$.
 \begin{lem}\label{lem1} Let $H$ be an irreducible subgroup of $\GL_n(\CX)$ and let $g\in \GL_n(k_0)$ normalize $H$.  Then $g \in \GL_n(\CX)\cdot \Scal_n(k_0)$.\end{lem}
\begin{proof} For any $h \in H$ and $g \in \GL_n(k_0)$ normalizing $H$, we have that \[0= \d_i(g^{-1}hg) = -g^{-1}\d_i(g)g^{-1}hg + g^{-1}h\d_i(g)\] for all $\d_i \in \Pi$.  Therefore, 
 \[\d_i(g)g^{-1} h = h \d_i(g)g^{-1}.\]
 Since $H$ is irreducible, Schur's Lemma implies that  $\d_i(g)g^{-1} = z_i\in \Scal_n(k_0)$.  One can check that the $z_i$ satisfy the integrability conditions  $\d_iz_j = \d_j z_i$ so there exists a nonzero $u \in \Scal_n(k_0)$ such that $\d_i u = z_i u$ for all $i$.  This implies that $\d_i(u^{-1}g) = 0$ for all $i$  and so $g = u h$ for some $h \in \GL_n(\CX)$.\end{proof}

It is well known that if $G$ and $H$ are linear algebraic groups with $H$ normal in $ G$, then $G/H$ is also a linear algebraic group.  For  $\Scal_n(k_0) \lhd \GL_n(k_0)$, we will denote by $\rho$ the canonical map $\rho: \GL_n(k_0) \rightarrow \GL_n(k_0)/\Scal_n(k_0)$. 
\begin{lem}\label{lem2} Let $H \subset \GL_n(k_0)$ be a Kolchin-connected linear differential algebraic group and let $\Hbar$ be its Zariski-closure in $\GL_n(k_0)$.  
Assume that  $\Hbar$ is  irreducible.  Then \[H \subset (H,H)_\Pi \cdot \Scal_n(k_0),\] where $(H,H)_\Pi$ is the Kolchin-closure of $(H,H) $.
\end{lem}
\begin{proof} Since $\Hbar$ is irreducible, it must be reductive (\cite{springer}, p.~37). Since $H$ is Kolchin-connected, $\Hbar$ is Zariski-connected  so we can write $\Hbar = Z(\Hbar)\cdot (\Hbar,\Hbar)$ where $Z(\Hbar)$ is the center of $\Hbar$ (\cite{humphreys}, Ch. 27.5).  Using the irreducibility again, Schur's Lemma implies that $Z(\Hbar) \subset \Scal_n(k_0)$. 
 Using the map $\rho$ above, we have that $\rho(\Hbar)$ is isomorphic to $ (\Hbar,\Hbar)/(Z(\Hbar)\cap (\Hbar,\Hbar))$ and so is a connected semisimple linear algebraic group.  Furthermore, $\rho(H)$ is a Zariski-dense,  Kolchin-connected, subgroup of $\rho(\Hbar)$. Propositions 11 and 13  of \cite{cassidy6} imply that $\rho(H)$ equals $(\rho(H),\rho(H))_\Pi$, the Kolchin-closure of its commutator subgroup $(\rho(H),\rho(H))$. Since $(H,H)_\Pi$ is a linear differential algebraic group, we have that $\rho((H,H)_\Pi)$ is a linear differential algebraic group containing $(\rho(H),\rho(H))$ and therefore contains $(\rho(H), \rho(H))_\Pi$. Since $(H,H)_\Pi \subset H$ we have that $\rho((H,H)_\Pi) = \rho(H)$. Therefore $H \subset (H,H)_\Pi \cdot \Scal_n(k_0)$. \end{proof}
 
 \begin{lem}\label{lem3} Let $G \subset \GL_n(k_0)$ be a linear differential group and assume that \begin{enumerate}
 \item $(G,G) \subset \GL_n(\CX)$ and 
 \item the identity component $\Gbar^0$ of $\Gbar$, the Zariski-closure of $G$ in $\GL_n(k_0)$, is irreducible.
 \end{enumerate}
 Then $G \subset \GL_n(\CX) \cdot \Scal_n(k_0)$.
 \end{lem}
\begin{proof} We first note that the Zariski-closure of $G^0$, the Kolchin-component of the identity of $G$ is Zariski-connected and of finite index in $\Gbar$. Therefore $\Gbar^0$ is the Zariski-closure $\overline{G^0}$ of $G^0$.  We now apply Lemma~\ref{lem2} to $H = G^0$ and conclude that $G^0 \subset (G^0,G^0)_\Pi \cdot \Scal_n(k_0)$.  Since $(G,G) \subset \GL_n(\CX)$ we have that $(G^0,G^0)_\Pi \subset \GL_n(\CX)$.  Furthermore, since $\Gbar^0$ is irreducible and is the Zariski-closure of $G^0$, we have that $G^0$ is irreducible.  Therefore $(G^0,G^0)_\Pi$ is an irreducible subgroup of $\GL_n(\CX)$. Any $g \in G$ normalizes $G^0$ and therefore normalizes $(G^0,G^0)_\Pi$. Applying Lemma~\ref{lem1} to $H = (G^0,G^0)_\Pi$, we have that $G \subset \GL_n(\CX) \cdot \Scal_n(k_0)$.\end{proof}

\begin{remark}\label{remark3.1} { Simple examples ({\it e.g.}, $G = \diag_n(k_0)$, the group of diagonal matrices) show that the condition $(G,G) \subset \GL_n(\CX)$ does not imply $G \subset \GL_n(\CX) \cdot \Scal_n(k_0)$ without some additional hypotheses. }
\end{remark}

\begin{prop}\label{prop3.2} Let $k,K,A,G$ be as in Proposition~\ref{prop3.1}. If Equation~(\ref{ppveqn}) is absolutely irreducible and $(G,G)$ is conjugate to a subgroup of  $\GL_n(\CX)$, then (\ref{ppveqn}) is projectively isomonodromic.
\end{prop}
\begin{proof} As noted above, $\Gbar$ is the usual Picard-Vessiot group of (\ref{ppveqn}) over $k$.  If (\ref{ppveqn}) is absolutely irreducible, then $\Gbar^0$ is an irreducible subgroup of $\GL_n(k_0)$. Lemma~\ref{lem3} implies that $G \subset \GL_n(\CX) \cdot \Scal_n(k_0)$
\end{proof}

\bibliographystyle{amsplain}

\end{document}